\documentclass[11pt,a4paper]{article}
\usepackage[margin=2.7cm]{geometry}
\usepackage{amsmath,amssymb,amsthm,mathtools,mathrsfs}
\usepackage{hyperref,enumitem}
\usepackage[T1]{fontenc}
\usepackage{lmodern}
\newcommand{\F}{\mathbb F}\newcommand{\G}{\mathbb G}\newcommand{\Pp}{\mathbb P}\newcommand{\E}{\mathbb E}\newcommand{\cS}{\mathcal S}\newcommand{\cH}{\mathcal H}\newcommand{\cT}{\mathcal T}
\newtheorem{theorem}{Theorem}[section]\newtheorem{proposition}[theorem]{Proposition}\newtheorem{lemma}[theorem]{Lemma}\theoremstyle{remark}\newtheorem{remark}[theorem]{Remark}\newtheorem{example}[theorem]{Example}\newtheorem{problem}[theorem]{Research problem}

\title{Backward and reflected backward stochastic differential equations under initial enlargement of Brownian filtrations}

\author{
	\textbf{Youssef Ouknine}\\[3pt]
	\small Affiliate Professor, Africa Business School,\\
	\small Mohammed VI Polytechnic University (UM6P),\\
	\small Rabat Campus, Rocade Rabat--Salé, 11103, Morocco\\
	\small Member, Hassan II Academy of Sciences and Technology,\\
	\small Km 4, Avenue Mohammed VI, Rabat, Morocco\\
	\small \url{youssef.ouknine@um6p.ma}\\
	\and
	\textbf{Badr Elmansouri}\\[3pt]
	\small Cadi Ayyad University,\\
	\small National School of Applied Sciences of Marrakech,\\
	\small BP 575, Avenue Abdelkrim Khattabi,\\
	\small 40000 Guéliz, Marrakech, Morocco\\
	\small \url{b.elmansouri@uca.ac.ma}
}

\date{\today}

\begin{document}\maketitle

\begin{abstract}
	We investigate the effect of initially enlarging
	a Brownian filtration by its terminal value
	on backward stochastic differential equations
	(BSDEs) and reflected BSDEs (RBSDEs).
	We establish the Brownian innovation
	representation and the predictable
	representation property in the enlarged
	filtration, highlighting the singular
	information drift arising from the
	semimartingale decomposition of the
	original Brownian motion.
	We derive transfer identities under
	suitable integrability conditions and
	provide explicit counterexamples showing
	that the resulting singular equations
	may fail to admit unique square-integrable
	solutions, including in the reflected case.
	We further develop a conditional
	parametrization of regular BSDEs and
	establish a conditional comparison
	principle for Snell envelopes, illustrating
	both the strict value of additional
	information and the failure of general
	pathwise domination.
	An exact projection identity is also
	obtained for zero-driver terminal payoffs.
	Finally, we discuss possible extensions
	to irregular obstacles and more general
	information structures.
\end{abstract}


\section{Introduction}

Backward stochastic differential equations (BSDEs)
provide a probabilistic framework for studying
stochastic control, nonlinear expectations, and
problems arising in mathematical finance.
Since the foundational work of Pardoux and Peng
\cite{PP}, existence and uniqueness results
have been established under suitable Lipschitz
and integrability assumptions, with further
extensions to $L^p$ solutions and financial
applications; see, for instance,
Briand et al.~\cite{BriandEtAl2003} and
El Karoui, Peng and Quenez
\cite{ElKarouiPengQuenez1997}.
A fundamental feature of these equations is
their dependence on the underlying filtration,
which determines both the available information
and the martingale representation used to
construct adapted solutions.

This dependence becomes particularly relevant
when the filtration is enlarged by additional
information. Such a situation arises naturally
when an agent has advance knowledge of a future
random outcome, while another agent observes
only the progressively revealed information.
The mathematical theory of enlargement of
filtrations provides a framework for describing
how this additional information affects
martingales and semimartingales; classical
references include Jeulin~\cite{Jeulin}
and Jacod~\cite{Jacod}.
However, the consequences for backward equations
require particular attention, since changing
the filtration may alter both the martingale
integrator and the adapted solution.

A complementary perspective concerns the
shrinkage of filtrations, where processes
adapted to a larger filtration $\G$ are
projected onto a smaller filtration $\F$.
Under suitable integrability assumptions,
the $\F$-optional projection of a $\G$-adapted process
$A_t=\int_0^t\vartheta_s\,ds$\footnote{Here, the process $\vartheta$ is a $\G$-adapted and integrable.} may satisfy
\[
\E[A_t\mid\mathcal F_t]
=
\int_0^t\E[\vartheta_s\mid\mathcal F_s]\,ds
+M_t,
\]
where $M$ is an $\F$-martingale.
Thus, projecting a finite-variation
process may introduce an additional
martingale component.
Related calculations and applications
to BSDEs are discussed in an unpublished
note by Jeanblanc~\cite{JeanblancPC}.

In this note, we consider the initial enlargement
of the completed natural filtration $\F$
of a Brownian motion $B$ by its terminal value
$B_T$, namely
$\mathcal G_t=\mathcal F_t\vee\sigma(B_T)$.
Although this enlargement provides complete
knowledge of the terminal Brownian position
at time zero, it does not eliminate the
randomness of the intermediate trajectory.
Indeed, the process
\[
W_t=B_t-\int_0^t\frac{B_T-B_s}{T-s}\,ds
\]
is a Brownian motion in $\G$, independent
of $B_T$, and the enlarged filtration admits
a predictable representation in terms of $W$.
On the other hand, the original Brownian
motion acquires the information drift
$\alpha_t=(B_T-B_t)/(T-t)$, which becomes
singular as $t$ approaches $T$.
This creates a distinction between BSDEs
formulated directly with respect to the
innovation process and equations obtained
by transferring the original Brownian
integrator to the enlarged filtration.

The distinction is important for well-posedness.
The classical BSDE theory applies to equations
driven by $W$ with uniformly Lipschitz
generators, whereas the transformed generator
$f(t,y,z)-z\alpha_t$ generally falls outside
this framework. In particular, the singular
information drift may lead to nonuniqueness,
even among square-integrable solutions for
which the drift integral is absolutely defined.
We illustrate this phenomenon by explicit
examples and establish transfer identities
under suitable integrability assumptions.
We also examine a conditional parametrization
of regular BSDEs, using the independence
between $B_T$ and $W$ to represent the
enlarged equation as a family of standard
Brownian BSDEs.

Reflected BSDEs introduce a further difficulty,
since the solution must satisfy an obstacle
constraint together with a minimality condition.
Their classical existence and uniqueness
theory is developed by El Karoui et al.
\cite{EKPQ}, with close connections to optimal
stopping and Snell envelopes.
In the enlarged filtration, the pathwise
reflection conditions may be preserved under
a suitable transfer, while the transformed
generator remains singular.
Moreover, additional terminal information
enlarges the class of admissible stopping
times but does not imply a general pathwise
comparison of the corresponding Snell
envelopes. We therefore investigate the
value of this information through conditional
expectations and provide explicit examples
illustrating both strict informational gains
and the failure of pathwise domination.

For irregular obstacles, the appropriate
framework involves optional strong
supermartingales and Mertens decompositions.
The works of Grigorova, Imkeller, Offen,
Ouknine and Quenez~\cite{GIOOQ}, and
Grigorova, Imkeller, Ouknine and
Quenez~\cite{GIOQ}, develop reflected BSDEs
and nonlinear optimal stopping problems
under weakened regularity assumptions.
Our complete reflected BSDE arguments are
restricted to continuous obstacles; the
extension to irregular optional obstacles
is discussed as a direction for further
research, particularly in relation to
singular information drifts and the
corresponding uniqueness questions.

The main emphasis of this paper is on the
distinction between regular BSDEs formulated
in the enlarged filtration and equations
obtained by transferring Brownian BSDEs
through the semimartingale decomposition
of the original driving noise. The terminal-information drift is singular
near maturity, and the resulting transferred
equations may fail to be unique even when
the corresponding equations in the original
filtration are well posed. We also examine the consequences of this
distinction for reflected equations,
optimal stopping, and financial claims
under asymmetric information.

The remainder of this paper is organized as follows.
Section \ref{sec2} introduces the enlarged Brownian
filtration and establishes the innovation and
martingale representation properties.
Section \ref{sec3} examines BSDEs under enlargement,
including transfer identities, well-posedness,
and nonuniqueness. Section \ref{sec4} develops a
conditional parametrization of regular BSDEs.
Section \ref{sec5} studies reflected BSDEs with continuous
obstacles, while Section \ref{sec6} investigates Snell
envelopes and the value of terminal information.
Section \ref{sec7} establishes an exact conditional
representation for endpoint payoffs.
Section \ref{sec:financial} presents financial
applications involving optimal exercise under
terminal information, reflected BSDE transfer,
and conditional valuation of European claims.
Finally, Section \ref{sec8} discusses irregular obstacles
and outlines directions for further research.

\section{The enlarged Brownian filtration}
\label{sec2}
We begin by introducing the Brownian
filtration enlargement that will serve
as the probabilistic framework throughout
this note. The main idea is to distinguish
between an observer who has access only
to the information progressively generated
by a Brownian motion and another observer
who possesses additional knowledge of its
terminal value from the initial time.

Such a situation arises naturally in
mathematical finance, particularly in models
of insider trading. Consider, for example,
a financial asset whose price dynamics are
driven by a Brownian motion. An ordinary
investor observes the evolution of the
market over time, whereas an informed
investor may possess advance knowledge
of a future market-related quantity.
In a simplified model, this additional
information can be represented by the
terminal Brownian value $B_T$.
Initial enlargement of filtrations provides
a natural mathematical description of
this information asymmetry; see
Jacod~\cite{Jacod}, Jeulin~\cite{Jeulin},
and the financial applications developed
by Amendinger, Imkeller and Schweizer
\cite{AmendingerImkellerSchweizer1998}.

More precisely, fix $T>0$ and let
$(\Omega,\mathcal F,\Pp)$ be a probability
space carrying a one-dimensional standard
Brownian motion $B=(B_t)_{0\le t\le T}$,
with $B_0=0$. Denote by
$\F=(\mathcal F_t)_{0\le t\le T}$
its completed natural filtration.

We introduce the initially enlarged filtration
$\G=(\mathcal G_t)_{0\le t\le T}$, defined by
\[
\mathcal G_t
=\mathcal F_t\vee\sigma(B_T),
\qquad 0\le t\le T,
\]
with the usual augmentation understood. 
In particular,
\[
\mathcal F_t\subseteq\mathcal G_t,
\qquad
\mathcal G_0=\sigma(B_T),
\]
up to completion.

Thus, $\mathcal F_t$ describes the
information available from observing
the Brownian path up to time $t$,
whereas $\mathcal G_t$ additionally
contains the terminal value $B_T$,
which is already known at time zero.
Although both filtrations are defined
on the same probability space, the
martingale properties of the original
Brownian motion are affected by this
additional information.

In particular, the process $B$ is no
longer a Brownian motion with respect
to $\G$. Its semimartingale decomposition
involves an information drift, given by
\begin{equation}\label{eq:alpha}
	\alpha_t=\frac{B_T-B_t}{T-t},
	\qquad
	W_t=B_t-\int_0^t\alpha_s\,ds,
	\quad 0\le t<T.
\end{equation}
The process $W$ will be shown to extend
to a Brownian motion in $\G$, independent
of $B_T$. The drift $\alpha$ captures
the effect of the additional terminal
information and becomes singular as
$t$ approaches $T$. Understanding its
integrability and the representation
properties of $W$ is therefore the
starting point for the analysis of
BSDEs under the enlarged filtration.

\begin{lemma}[Integrability of the information drift]
	\label{lem:drift}
	Almost surely $\int_0^T|\alpha_s|\,ds<\infty$, while
	$\E\int_0^T|\alpha_s|^2\,ds=\infty$.
	More precisely, for every $t<T$,
	$\E|\alpha_t|=\sqrt{2/\pi}(T-t)^{-1/2}$ and
	$\E[\alpha_t^2]=(T-t)^{-1}$.
\end{lemma}

\begin{proof}
	Since $B_T-B_t$ is centered Gaussian with variance
	$T-t$, we have $\alpha_t\overset{d}{=}N/\sqrt{T-t}$,
	where $N$ is a standard normal random variable.
	Using $\E|N|=\sqrt{2/\pi}$ and $\E[N^2]=1$, we obtain
	\[
	\E|\alpha_t|
	=\sqrt{\frac{2}{\pi}}(T-t)^{-1/2},
	\qquad
	\E[\alpha_t^2]=(T-t)^{-1}.
	\]
	Since the process $(|\alpha_t|)_{0 \leq t \leq T}$ is nonnegative, Tonelli's theorem
	yields
	\[
	\E\int_0^T|\alpha_s|\,ds
	=\int_0^T\E|\alpha_s|\,ds
	=2\sqrt{\frac{2T}{\pi}}<\infty.
	\]
	As the random variable $\int_0^T|\alpha_s|\,ds$
	is nonnegative and has finite expectation, it is
	finite almost surely.
	
	Similarly, applying Tonelli's theorem to the
	nonnegative process $(|\alpha_t|^2)_{0 \leq t \leq T}$, we obtain
	\[
	\E\int_0^T|\alpha_s|^2\,ds
	=\int_0^T\frac{1}{T-s}\,ds
	=\infty,
	\]
	which completes the proof.
\end{proof}

\begin{theorem}[Innovation and representation]
	\label{thm:PRP}
	The process $W$ in \eqref{eq:alpha} extends continuously
	to $T$ and is a $\G$-Brownian motion. It is independent
	of $B_T$, and
	\begin{equation}\label{eq:recovery}
		B_t=\frac{t}{T}B_T
		+(T-t)\int_0^t\frac{dW_s}{T-s},
		\quad 0\le t<T.
	\end{equation}
	Consequently,
	$\mathcal G_t=\sigma(B_T)\vee\mathcal F_t^W$
	up to completion for $t<T$, and also at $T$.
	Moreover, every $\xi\in L^2(\mathcal G_T)$ admits
	the unique representation
	\begin{equation}\label{eq:PRP}
		\xi=\E[\xi\mid B_T]+\int_0^T H_s\,dW_s,
		\qquad H\in\cH^2(\G).
	\end{equation}
\end{theorem}

\begin{proof}
	We first show that $W$ is a $\G$-Brownian motion.
	For $0\le s\le u<T$, the Gaussian regression
	formula, applied to the jointly Gaussian
	increments $B_T-B_u$ and $B_T-B_s$, gives
	\[
	\E[B_T-B_u\mid\mathcal G_s]
	=\frac{T-u}{T-s}(B_T-B_s).
	\]
	
	It follows that, for $s\le u<T$,
	$\E[\alpha_u\mid\mathcal G_s]
	=(B_T-B_s)/(T-s)$. Hence, by conditional
	Fubini's theorem, for $s<t<T$,
	\[
	\E\left[\int_s^t\alpha_u\,du
	\Bigm|\mathcal G_s\right]
	=\int_s^t\E[\alpha_u\mid\mathcal G_s]\,du
	=\frac{t-s}{T-s}(B_T-B_s).
	\]
	It follows that
	$\E[W_t-W_s\mid\mathcal G_s]=0$.
	Thus $W$ is a continuous $\G$-martingale on
	every compact subinterval of $[0,T)$.
	Since $\int_0^t\alpha_s\,ds$ is continuous and has finite
	variation (from Lemma \ref{lem:drift}) and $[B]_t=t$, we have $[W]_t=t$.
	By L\'evy's characterization, $W$ is a
	$\G$-Brownian motion on $[0,T)$ (see Theorem 39 in \cite[Ch II]{Protter2004}).
	
	By Lemma~\ref{lem:drift},
	$\int_0^T|\alpha_s|\,ds<\infty$ almost surely.
	Therefore, the continuity of $B$ implies that
	$W_t$ converges almost surely as $t\uparrow T$ to
	$W_T=B_T-\int_0^T\alpha_s\,ds$. Moreover, for $s<t<T$, the increment $W_t-W_s$
is independent of $\mathcal G_s$ and has distribution
$\mathcal N(0,t-s)$. Since $W_t\to W_T$ almost surely
as $t\uparrow T$, the convergence of Gaussian laws
and the preservation of independence under almost
sure convergence imply that $W_T-W_s$ is independent
of $\mathcal G_s$ and has distribution
$\mathcal N(0,T-s)$. Hence, $W$ is a
$\G$-Brownian motion on $[0,T]$.
Since $B_T$ is $\mathcal G_0$-measurable,
the Brownian property also implies that
$W$ is independent of $B_T$.

	Next, setting $U_t=B_T-B_t$, we obtain
	$dU_t=-dW_t-U_t(T-t)^{-1}dt$.
	Applying the product rule to $U_t/(T-t)$
	gives
	$$
	d\big(\frac{U_t}{T-t}\big)
	=-\frac{dW_t}{T-t}.
	$$
	Since $U_0=B_T$, integration yields
	\eqref{eq:recovery}. As $W$ is constructed
	from $B$ and $B_T$, we have
	$\sigma(B_T)\vee\mathcal F_t^W
	\subseteq\mathcal G_t$.
	The reverse inclusion follows from
	\eqref{eq:recovery}. The identity at $T$
	follows by continuity and completion.

	It remains to establish the predictable
	representation property.\footnote{
		For related results on martingale representation
		under initial enlargement of filtrations,
		see \cite{Amendinger200}.
	} By the filtration identity,
	$\mathcal G_T=\sigma(B_T)\vee\mathcal F_T^W$.
	Since $W$ and $B_T$ are independent,
	finite linear combinations of random
	variables of the form
	$\xi=h(B_T)\eta$, with bounded measurable
	$h$ and $\eta\in L^2(\mathcal F_T^W)$,
	are dense in $L^2(\mathcal G_T)$.
	
	By the classical Brownian martingale
	representation theorem, we may write
	$\eta=\E[\eta]+\int_0^T H_s^\eta\,dW_s$,
	where $H^\eta$ is $\mathcal F^W$-predictable
	and square-integrable. Thus
	\[
	\xi=\E[\xi\mid B_T]
	+\int_0^T h(B_T)H_s^\eta\,dW_s.
	\]
	The integrand is $\G$-predictable and
	belongs to $\cH^2(\G)$.
	
	For an arbitrary $\xi\in L^2(\mathcal G_T)$,
	choose a sequence of such finite linear
	combinations converging to $\xi$ in $L^2$.
	The conditional expectation is an
	orthogonal projection in $L^2$, and
	the It\^o isometry shows that the
	corresponding integrands form a Cauchy
	sequence in $\cH^2(\G)$. Passing to the
	limit gives \eqref{eq:PRP}.
	
	Finally, if two integrands yield the
	same representation, the It\^o isometry
	implies that their difference vanishes
	in $\cH^2(\G)$. This proves uniqueness.
\end{proof}

\begin{remark}
	Unlike the classical Brownian martingale representation,
	the initial term in \eqref{eq:PRP} need not be deterministic,
	since $\mathcal G_0=\sigma(B_T)$ up to completion.
	For instance, $\E[B_T\mid\mathcal G_t]=B_T$ for every
	$t\in[0,T]$, whereas
	$\E[B_T\mid\mathcal F_t]=B_t$.
	Thus, the same terminal random variable gives rise
	to different martingales in the two filtrations.
\end{remark}

\section{BSDEs and transfer between filtrations}
\label{sec3}
Write $\cS^2(\mathbb H)$ for the space of continuous
$\mathbb H$-adapted processes $Y$ satisfying
$\E\sup_{t\le T}|Y_t|^2<\infty$, and
$\cH^2(\mathbb H)$ for the space of
$\mathbb H$-predictable processes $Z$ satisfying
$\E\int_0^T|Z_s|^2\,ds<\infty$.

Let $\xi\in L^2(\mathcal F_T)$ and let $f$ be
$\F$-progressively measurable, uniformly Lipschitz
in $(y,z)$, with
$\E(\int_0^T|f(s,0,0)|\,ds)^2<\infty$.
Under these standard assumptions, the BSDE
\begin{equation}\label{eq:FBSDE}
	Y_t^\F=\xi+\int_t^T f(s,Y_s^\F,Z_s^\F)\,ds
	-\int_t^T Z_s^\F\,dB_s
\end{equation}
admits a unique solution in
$\cS^2(\F)\times\cH^2(\F)$ 
(see Pardoux and Peng~\cite{PP}).
.

\begin{proposition}[Transfer identity]
	\label{prop:transfer}
	Suppose $(Y^\F,Z^\F)$ solves \eqref{eq:FBSDE}
	and satisfies
	\begin{equation}\label{eq:transfercondition}
		\int_0^T|Z_s^\F\alpha_s|\,ds<\infty
		\quad\text{almost surely}.
	\end{equation}
	Then the same pair satisfies, in the enlarged
	filtration $\G$,
	\begin{equation}\label{eq:transferred}
		Y_t^\F=\xi+\int_t^T
		\big[f(s,Y_s^\F,Z_s^\F)-Z_s^\F\alpha_s\big]\,ds
		-\int_t^T Z_s^\F\,dW_s.
	\end{equation}
	The identity holds up to indistinguishability.
	No uniqueness for \eqref{eq:transferred}
	is asserted.
\end{proposition}

\begin{proof}
	By Theorem~\ref{thm:PRP}, we have
	$B_t=W_t+\int_0^t\alpha_s\,ds$ as
	$\G$-semimartingales. Since $\F\subseteq\G$,
	the process $Z^\F$ is also $\G$-predictable.
	Hence, by the decomposition of stochastic
	integrals with respect to semimartingales,
	for every $t<T$,
	$$
	\int_0^t Z_s^\F\,dB_s
	=\int_0^t Z_s^\F\,dW_s
	+\int_0^t Z_s^\F\alpha_s\,ds.
	$$
	Moreover, $Z^\F\in\cH^2(\G)$, so the stochastic
	integral with respect to $W$ extends to $T$.
	The finite-variation integral also converges
	absolutely by \eqref{eq:transfercondition}.
	Passing to the limit as $t\uparrow T$ therefore
	gives the same identity at $T$.
	
	Substituting this decomposition into
	\eqref{eq:FBSDE} yields \eqref{eq:transferred}
	for all $t\in[0,T]$, up to indistinguishability.

	No uniqueness is asserted for \eqref{eq:transferred},
	since the transformed generator
	$\widetilde f(t,y,z)=f(t,y,z)-z\alpha_t$
	does not satisfy the standard uniform
	Lipschitz condition. In fact, uniqueness
	may fail even in
	$\cS^2(\G)\times\cH^2(\G)$,
	as shown in Theorem~\ref{thm:nonuniq}.
\end{proof}

\begin{remark}[A useful sufficient condition]
	If $|Z_t^{\F}|\le C$ almost everywhere,
	then \eqref{eq:transfercondition} holds by
	Lemma~\ref{lem:drift}, since
	\[
	\int_0^T|Z_s^{\F}\alpha_s|\,ds
	\le C\int_0^T|\alpha_s|\,ds<\infty
	\quad\text{almost surely}.
	\]
	The assumption $Z^{\F}\in\cH^2(\F)$ alone
	does not ensure this integrability.
	Indeed, the usual Cauchy--Schwarz argument
	is insufficient, since
	$\E\int_0^T|\alpha_s|^2\,ds=\infty$ by Lemma~\ref{lem:drift}.
\end{remark}

\begin{proposition}[Well-posedness of a regular enlarged BSDE]
	Let $\xi\in L^2(\mathcal G_T)$ and let $g$ be
	$\G$-progressively measurable, uniformly Lipschitz
	in $(y,z)$, with
	$\E(\int_0^T|g(s,0,0)|\,ds)^2<\infty$.
	Then the BSDE
	\begin{equation}\label{eq:GBSDE}
		Y_t^{\G}=\xi+\int_t^Tg(s,Y_s^{\G},Z_s^{\G})\,ds
		-\int_t^TZ_s^{\G}\,dW_s
	\end{equation}
	admits a unique solution in
	$\cS^2(\G)\times\cH^2(\G)$.
\end{proposition}

\begin{proof}
	By Theorem~\ref{thm:PRP}, $W$ is a
	$\G$-Brownian motion with the predictable
	representation property. The classical
	existence and uniqueness argument for
	Lipschitz BSDEs, based on Picard iteration
	and a priori estimates (see
	Pardoux and Peng~\cite{PP},
	El Karoui et al.~\cite{ElKarouiPengQuenez1997},
	and Briand et al.~\cite{BriandEtAl2003}),
	therefore extends to the present filtration.
	Indeed, the martingale representation
	property and the usual stochastic integral
	estimates remain valid in $\G$.
	The possibly nontrivial initial sigma-field
	$\mathcal G_0$ does not affect the contraction
	argument, and uniqueness follows from
	the standard stability estimates.
\end{proof}

\begin{example}[Different information, different solutions]
	\label{ex:linear}
	Let $\xi=B_T$ and $f=g=0$. In the natural
	filtration $\F$, the unique solution is
	$(Y^{\F},Z^{\F})=(B,1)$, since
	$\E[B_T\mid\mathcal F_t]=B_t$.
	In contrast, as $B_T$ is $\mathcal G_0$-measurable,
	the unique $\G$-solution driven by $W$ is
	$(Y^{\G},Z^{\G})=(B_T,0)$.

	More generally, for any Borel function $h$
	such that $h(B_T)\in L^2$, the zero-driver
	solutions satisfy
	\[
	Y_t^{\F}=\E[h(B_T)\mid\mathcal F_t],
	\qquad
	Y_t^{\G}=h(B_T).
	\]
	If $h$ is bounded and measurable, then
	$Y_t^{\F}=P_{T-t}h(B_t)$, where
	$P_sh(x)=\E[h(x+B_s)]$ denotes the
	Brownian heat semigroup (see e.g. \cite{KaratzasShreve1991}).
\end{example}

\begin{theorem}[Nonuniqueness for the transferred singular generator]
	\label{thm:nonuniq}
	For the terminal condition $\xi=B_T$, the BSDE
	in the enlarged filtration $\G$,
	\begin{equation}\label{eq:singular}
		Y_t^{\G}
		=B_T-\int_t^T Z_s^{\G}\alpha_s\,ds
		-\int_t^T Z_s^{\G}\,dW_s,
	\end{equation}
	admits at least two distinct solutions in
	$\cS^2(\G)\times\cH^2(\G)$, with absolutely
	integrable drift, namely
	\[
	(Y^{\G,1},Z^{\G,1})=(B,1),
	\qquad
	(Y^{\G,2},Z^{\G,2})=(B_T,0).
	\]
\end{theorem}

\begin{proof}
	By the semimartingale decomposition established
	in Theorem~\ref{thm:PRP}, we have
	\[
	B_T-B_t=\int_t^T\alpha_s\,ds+W_T-W_t.
	\]
	Rearranging this identity shows that
	$(Y^{\G,1},Z^{\G,1})=(B,1)$ satisfies
	\eqref{eq:singular}. The second pair
	$(Y^{\G,2},Z^{\G,2})=(B_T,0)$ also satisfies
	the equation, since both integrals vanish
	and $B_T$ is $\mathcal G_0$-measurable.
	
	Both pairs belong to
	$\cS^2(\G)\times\cH^2(\G)$.
	Indeed, $B$ is continuous and $\G$-adapted,
	and Doob's maximal inequality, applied
	in the natural filtration $\F$, yields
	$\E\sup_{t\le T}|B_t|^2\le 4T$.
	Moreover,
	$\E|B_T|^2=T$ and
	$\E\int_0^T1^2\,ds=T$.	Moreover, the drift integrals are absolutely
	convergent, since
	$\int_0^T|\alpha_s|\,ds<\infty$ almost surely
	by Lemma~\ref{lem:drift} for the first solution,
	while the drift vanishes for the second.

	Finally, for every deterministic $t<T$,
	$B_T-B_t\sim\mathcal N(0,T-t)$, so
	$\mathbb P(B_T\ne B_t)=1$.
	Hence the two solutions are distinct,
	which proves nonuniqueness.
\end{proof}

\begin{remark}
	This result does not contradict the classical
	uniqueness theorem for uniformly Lipschitz BSDEs.
	Indeed, the transformed generator
	$\widetilde f(t,y,z)=-z\alpha_t$
	has the random Lipschitz coefficient
	$|\alpha_t|$, which cannot be bounded by
	a deterministic constant uniformly in
	$t\in[0,T)$.
	Thus, the standard uniqueness theorem
	is not applicable. In particular,
	square integrability of $(Y^{\G},Z^{\G})$
	and absolute integrability of the drift
	do not suffice to ensure uniqueness
	for the transferred equation.
\end{remark}

\section{Conditional decomposition and parametrization}
\label{sec4}
The product structure of the enlarged filtration
provides a convenient framework for studying
BSDEs. Let $\mu$ denote the law of $B_T$ and
let $(\Omega_W,\mathcal F^W,\mathbb P_W)$ be
the canonical Wiener space, with
$\Omega_W=C_0([0,T],\mathbb R)$.
By Theorem~\ref{thm:PRP}, the random variable
$B_T$ and the innovation process $W$ are
independent, and
$\mathcal G_t=\sigma(B_T)\vee\mathcal F_t^W$
up to completion. Consequently, the pair
$(B_T,W)$ has law $\mu\otimes\mathbb P_W$,
and we may work on the product probability
space
\[
(\overline\Omega,\overline{\mathcal F},
\overline{\mathbb P})
=
(\mathbb R\times\Omega_W,\,
\mathcal B(\mathbb R)\otimes\mathcal F^W,\,
\mu\otimes\mathbb P_W),
\]
equipped with the completed filtration generated
by the two canonical coordinates.

This representation separates the initial
information $B_T$ from the Brownian innovation
$W$. In particular, if $g(t,\omega,y,z)$ is
jointly measurable, $\G$-progressive in
$(t,\omega)$, and uniformly Lipschitz in
$(y,z)$, 
By the measurable functional representation
principle (see Kallenberg~\cite[Lemma~1.13]
{Kallenberg2002}), together with the
product structure of the progressive
sigma-field, the driver admits a jointly
measurable parametrization
$g^x(t,\omega_W,y,z)$ such that
$$
g(t,\omega,y,z)
=g^{B_T(\omega)}(t,W(\omega),y,z),
$$
up to the appropriate null-set identifications.
For $\mu$-almost every $x$, the process
$g^x$ is progressively measurable with
respect to $\F^W$; the Brownian filtration generated by $W$.

The measurable factorization used here is
a consequence of standard results on
measurable mappings and product spaces
(see Kallenberg~\cite{Kallenberg2002}).
For background on enlargement of filtrations,
see Jacod~\cite{Jacod} and
Protter~\cite{Protter2004}.

We denote by $\E_W$ the expectation with
respect to the Wiener measure $\mathbb P_W$,
with the parameter $x$ held fixed.

\begin{proposition}[Disintegration of regular BSDEs]
	\label{prop:disintegrate}
	Suppose that the terminal condition and the
	generator admit jointly measurable representations
	$\xi=\xi^{B_T}(W)$ and
	$g(t,\omega,y,z)=g^{B_T}(t,W,y,z)$,
	where, for $\mu$-almost every $x$,
	$\xi^x$ is $\mathcal F_T^W$-measurable and
	$g^x$ is $\F^W$-progressively measurable.
	Assume that $g^x$ is uniformly Lipschitz
	in $(y,z)$, with a constant independent of $x$,
	and that
	\[
	\int_{\mathbb R}\E_W\Big[
	|\xi^x|^2+
	\Big(\int_0^T|g^x(s,0,0)|\,ds\Big)^2
	\Big]\,\mu(dx)<\infty.
	\]
	Then, for $\mu$-almost every $x$, the Brownian
	BSDE with data $(\xi^x,g^x)$ admits a unique
	solution $(Y^x,Z^x)$ in
	$\cS^2(\F^W)\times\cH^2(\F^W)$.
	The solutions can be chosen jointly measurable
	in the parameter and the Brownian variables,
	with $Z^x$ jointly predictable, and
	$$
	(Y^\G,Z^\G)=(Y^{B_T},Z^{B_T})
	$$
	is the unique solution of \eqref{eq:GBSDE}.
\end{proposition}

\begin{proof}
	By the integrability assumption and Tonelli's
	theorem, the data $(\xi^x,g^x)$ satisfy the
	standard Lipschitz BSDE hypotheses for
	$\mu$-almost every $x$. Hence, the corresponding
	Brownian BSDE admits a unique solution
	$(Y^x,Z^x)$.

To construct a jointly measurable family, we use
Picard iteration on the product space introduced
above. Starting from
$(Y^{x,0},Z^{x,0})=(0,0)$, define, for $n\geq 0$,
\[
Y_t^{x,n+1}
=\xi^x+\int_t^T
g^x(s,Y_s^{x,n},Z_s^{x,n})\,ds
-\int_t^T Z_s^{x,n+1}\,dW_s.
\]
More precisely, setting
\[
\eta^{x,n}
=\xi^x+\int_0^T
g^x(s,Y_s^{x,n},Z_s^{x,n})\,ds,
\]
the Brownian martingale representation theorem
gives
$\E_W[\eta^{x,n}\mid\mathcal F_t^W]
=\E_W[\eta^{x,n}]
+\int_0^t Z_s^{x,n+1}\,dW_s$,
and $Y^{x,n+1}$ is then obtained from
\[
Y_t^{x,n+1}
=\E_W[\eta^{x,n}\mid\mathcal F_t^W]
-\int_0^t
g^x(s,Y_s^{x,n},Z_s^{x,n})\,ds.
\]

The joint measurability of the data ensures
that the conditional expectations can be chosen
jointly measurable in $(x,t,\omega_W)$.
The corresponding martingale integrands can
also be chosen jointly predictable by the
parametrized martingale representation
argument, using approximation by elementary
predictable processes and the It\^o isometry.

Since the Lipschitz constant is independent
of $x$, the standard BSDE estimates yield
convergence of the Picard iterates in the
integrated norm
\[
\int_{\mathbb R}\E_W\left[
\sup_{0\le t\le T}|Y_t^{x,n}-Y_t^x|^2
+\int_0^T|Z_s^{x,n}-Z_s^x|^2\,ds
\right]\mu(dx)\longrightarrow 0.
\]
Passing to suitable subsequences and
measurable versions yields a jointly
measurable family $(Y^x,Z^x)$, with
$Z^x$ jointly predictable.

	By the standard a priori estimates for Lipschitz
	BSDEs, there exists a constant $C>0$,
	independent of $x$, such that
	\[
	\E_W\left[\sup_{t\le T}|Y_t^x|^2
	+\int_0^T|Z_s^x|^2\,ds\right]
	\le C\E_W\left[|\xi^x|^2+
	\left(\int_0^T|g^x(s,0,0)|\,ds\right)^2\right].
	\]
	Integrating with respect to $\mu$ and using
	the integrability assumption, we obtain
	$$
	\int_{\mathbb R}\E_W\int_0^T
	|Z_s^x|^2\,ds\,\mu(dx)<\infty.
	$$
	
	Since $B_T$ and $W$ are independent, their joint
	law is $\mu\otimes\mathbb P_W$. Hence, by
	Tonelli's theorem,
	\[
	\E\int_0^T|Z_s^{B_T}|^2\,ds
	=
	\int_{\mathbb R}\E_W\int_0^T
	|Z_s^x|^2\,ds\,\mu(dx)<\infty,
	\]
	where the finiteness follows from the standard
	a priori estimates and the integrability
	assumption. Thus $Z^{B_T}\in\cH^2(\G)$.
	
	It remains to verify that the parametrized
	solutions yield a solution in $\G$.
	For elementary jointly predictable integrands,
	evaluation at $x=B_T$ commutes with stochastic
	integration. By approximation in the integrated
	$\cH^2$ norm and the It\^o isometry, this
	property extends to $Z^x$. Consequently,
	\[
	\left.\int_0^t Z_s^x\,dW_s\right|_{x=B_T}
	=
	\int_0^t Z_s^{B_T}\,dW_s,
	\qquad 0\le t\le T,
	\]
	up to indistinguishability.
	
	Evaluating the parametrized BSDE at $x=B_T$
	and using the representations of $\xi$ and $g$,
	we obtain
	\[
	Y_t^{B_T}
	=\xi+\int_t^T
	g(s,Y_s^{B_T},Z_s^{B_T})\,ds
	-\int_t^T Z_s^{B_T}\,dW_s.
	\]
	Therefore, $(Y^{B_T},Z^{B_T})$ solves
	\eqref{eq:GBSDE}. Uniqueness follows from
	the well-posedness result for regular BSDEs
	in $\G$.

\end{proof}

\begin{remark}
	Conditionally on $B_T=x$, the original process
	$B$ has the law of a Brownian bridge from
	$0$ to $x$, with conditional mean
	$\E[B_t\mid B_T=x]=tx/T$ and covariance
	$\operatorname{Cov}(B_s,B_t\mid B_T=x)
	=\min(s,t)-st/T$.
	In contrast, since $W$ is independent of
	$B_T$, its conditional law given $B_T=x$
	remains the Wiener measure $\mathbb P_W$
	for $\mu$-almost every $x$.
	Thus, the parametrization in
	Proposition~\ref{prop:disintegrate}
	is performed with respect to the innovation
	$W$, allowing the use of standard Brownian
	BSDE theory for each fixed parameter $x$.
\end{remark}

\section{Reflected BSDEs with continuous obstacles}
\label{sec5}
Let $L$ be a continuous, $\F$-adapted process
satisfying $\E\sup_{t\le T}(L_t^+)^2<\infty$
and $L_T\le\xi$ almost surely.
Assume that $\xi\in L^2(\mathcal F_T)$ and
that $f$ satisfies the standard Lipschitz
and integrability conditions introduced above.
By the classical existence and uniqueness
result for reflected BSDEs
(see El Karoui et al.~\cite{ElKarouiPengQuenez1997}),
there exists a unique triple
$(Y^{\F},Z^{\F},K^{\F})$, where
$(Y^{\F},Z^{\F})\in
\cS^2(\F)\times\cH^2(\F)$ and
$K^{\F}$ is continuous, adapted,
nondecreasing, with $K_0^{\F}=0$ and
$\E[(K_T^{\F})^2]<\infty$, such that
\begin{align}
	Y_t^{\F}
	&=\xi+\int_t^T f(s,Y_s^{\F},Z_s^{\F})\,ds
	+K_T^{\F}-K_t^{\F}
	-\int_t^T Z_s^{\F}\,dB_s,
	\label{eq:FRBSDE}\\
	Y_t^{\F}&\ge L_t,\qquad
	\int_0^T(Y_t^{\F}-L_t)\,dK_t^{\F}=0.
	\nonumber
\end{align}

\begin{theorem}[Transfer of reflection]
	\label{thm:reflecttransfer}
	Suppose that $(Y^{\F},Z^{\F},K^{\F})$
	solves \eqref{eq:FRBSDE} and satisfies
	$\int_0^T|Z_s^{\F}\alpha_s|\,ds<\infty$
	almost surely. Then the same triple,
	viewed in $\G$, satisfies the reflected
	BSDE driven by $W$, with terminal condition
	$\xi$, obstacle $L$, and generator
	$\widetilde f(t,y,z)=f(t,y,z)-z\alpha_t$.
	More precisely,
	\begin{align}
		Y_t^{\F}
		&=\xi+\int_t^T
		\widetilde f(s,Y_s^{\F},Z_s^{\F})\,ds
		+K_T^{\F}-K_t^{\F}
		-\int_t^T Z_s^{\F}\,dW_s,
		\label{eq:GRBSDE}\\
		Y_t^{\F}&\ge L_t,\qquad
		\int_0^T(Y_t^{\F}-L_t)\,dK_t^{\F}=0.
		\nonumber
	\end{align}
	The transformed drift integral is understood
	in the almost surely absolutely convergent
	sense.
\end{theorem}

\begin{proof}
	By Theorem~\ref{thm:PRP}, we have
	$B_t=W_t+\int_0^t\alpha_s\,ds$.
	Since $\F\subseteq\G$, the processes
	$Y^{\F}$ and $K^{\F}$ remain $\G$-adapted,
	while $Z^{\F}$ is $\G$-predictable and
	belongs to $\cH^2(\G)$.
	
	As in Proposition~\ref{prop:transfer},
	the semimartingale decomposition gives
	\[
	\int_t^T Z_s^{\F}\,dB_s
	=\int_t^T Z_s^{\F}\,dW_s
	+\int_t^T Z_s^{\F}\alpha_s\,ds.
	\]
	The stochastic integral with respect to
	$W$ is well defined, and the finite-variation
	integral converges absolutely by assumption.
	Substituting into \eqref{eq:FRBSDE} yields
	\eqref{eq:GRBSDE}.
	
	Moreover, the obstacle constraint
	$Y^{\F}\ge L$ is preserved since it holds
	pathwise. The process $K^{\F}$ remains
	continuous and nondecreasing, and the
	Skorokhod condition
	$\int_0^T(Y_t^{\F}-L_t)\,dK_t^{\F}=0$
	is a pathwise Stieltjes identity.
	Hence it is unchanged by the enlargement
	of filtration, completing the proof.
\end{proof}

\begin{example}[Nonuniqueness of the transferred reflected BSDE]
	\label{ex:reflectednonuniq}
	Let $\xi=B_T$, $f=0$, and consider the continuous
	$\F$-adapted obstacle $L_t=B_t$. The transferred
	reflected BSDE in $\G$ is
	\[
	Y_t=B_T-\int_t^T Z_s\alpha_s\,ds
	+K_T-K_t-\int_t^T Z_s\,dW_s,
	\]
	subject to $Y_t\ge B_t$ and
	$\int_0^T(Y_t-B_t)\,dK_t=0$.
	
	One solution is $(Y^1,Z^1,K^1)=(B,1,0)$.
	A second solution is given by
	\[
	Y_t^2=B_t+(B_t-B_T)^2+T-t,\qquad
	Z_t^2=1+2(B_t-B_T),\qquad K_t^2=0.
	\]
	Indeed, It\^o's formula yields
	$dY_t^2=Z_t^2\,dB_t
	=Z_t^2\,dW_t+Z_t^2\alpha_t\,dt$.
	Since $Y_T^2=B_T$, the second triple satisfies
	the transferred equation. Moreover,
	$Y_t^2-L_t=(B_t-B_T)^2+T-t\ge0$, and
	the Skorokhod condition holds since $K^2=0$.
	
	Both triples belong to the usual
	square-integrable solution spaces.
	The drift integrals are absolutely convergent
	by Lemma~\ref{lem:drift}, since both integrands
	$Z^1$ and $Z^2$ have bounded trajectories
	on $[0,T]$. Finally,
	$Y_t^2-Y_t^1=(B_t-B_T)^2+T-t>0$ for every
	$t<T$. Thus the transferred reflected BSDE
	admits at least two distinct solutions.
\end{example}

\begin{example}[Nonuniqueness with active reflection]
	Let $T=1$, $\xi=B_1$, $f=0$, and consider
	the continuous $\F$-adapted obstacle
	$L_t=B_t+c(1-t)$, where $0<c<1$.
	The transferred reflected BSDE is
	\[
	Y_t=B_1-\int_t^1Z_s\alpha_s\,ds
	+K_1-K_t-\int_t^1Z_s\,dW_s.
	\]
	
	A first solution is
	\[
	Y_t^1=B_t+c(1-t),\qquad
	Z_t^1=1,\qquad K_t^1=ct.
	\]
	Indeed, $dY_t^1=dB_t-c\,dt
	=Z_t^1\,dB_t-dK_t^1$, and
	$Y^1=L$, so the Skorokhod minimality
	condition holds.
	
	A second solution is given by
	\[
	Y_t^2=B_t+(B_t-B_1)^2+1-t,\qquad
	Z_t^2=1+2(B_t-B_1),\qquad K_t^2=0.
	\]
	By It\^o's formula,
	$dY_t^2=Z_t^2\,dB_t$ and $Y_1^2=B_1$.
	Moreover,
	\[
	Y_t^2-L_t
	=(B_t-B_1)^2+(1-c)(1-t)\ge0.
	\]
	Thus the obstacle and minimality
	conditions are satisfied.
	
	Using $dB_t=dW_t+\alpha_t\,dt$,
	both triples satisfy the transferred
	reflected equation. They belong to the
	usual square-integrable solution spaces,
	and their drift integrals are absolutely
	convergent, since the corresponding
	integrands have bounded trajectories
	and $\int_0^1|\alpha_s|\,ds<\infty$
	almost surely.
	
	Since $K_t^1=ct$ whereas $K_t^2=0$,
	the two solutions are distinct.
	Hence nonuniqueness may occur even
	when the reflection mechanism is
	active for one of the solutions.
\end{example}

\begin{remark}
	Theorem~\ref{thm:reflecttransfer} establishes
	a transfer identity, not a general
	existence-and-uniqueness result for the
	transformed reflected BSDE.
	Indeed, the generator
	$\widetilde f(t,y,z)=f(t,y,z)-z\alpha_t$
	does not satisfy the standard uniform
	Lipschitz condition. Thus uniqueness
	cannot be inferred from the classical
	reflected BSDE theory.
\end{remark}

\section{Snell envelopes and the value of terminal information}
\label{sec6}
Let $X$ be a continuous, $\F$-adapted process
satisfying $\E\sup_{t\le T}|X_t|^2<\infty$.
For $\mathbb H\in\{\F,\G\}$ and $t\in[0,T]$,
let $\cT_{t,T}^{\mathbb H}$ denote the set of
all $\mathbb H$-stopping times taking values
in $[t,T]$. Define the Snell envelope
\[
S_t^{\mathbb H}
=\operatorname*{ess\,sup}_{\tau\in
	\cT_{t,T}^{\mathbb H}}
\E[X_\tau\mid\mathcal H_t].
\]
Under the present assumptions, these Snell
envelopes admit continuous versions.
In the zero-driver reflected setting,
$S^{\mathbb H}$ is the first component
of the corresponding reflected BSDE with
terminal condition $X_T$ and obstacle $X$
(see El Karoui et al.~\cite{ElKarouiPengQuenez1997}).

\begin{theorem}[Conditional value of information]
	\label{thm:value}
	For every deterministic $t\in[0,T]$,
	\begin{equation}\label{eq:value}
		\E[S_t^{\G}\mid\mathcal F_t]
		\ge S_t^{\F}
		\qquad \Pp\text{-a.s.}
	\end{equation}
	In particular, $\E [S_0^{\G}]\ge\E [S_0^{\F}]$.
\end{theorem}

\begin{proof}
	Since $\F\subseteq\G$, every $\F$-stopping
	time is also a $\G$-stopping time. Thus,
	for any $\tau\in\cT_{t,T}^{\F}$,
	$$
	S_t^{\G}
	\ge\E[X_\tau\mid\mathcal G_t].
	$$
	Taking conditional expectations with respect
	to $\mathcal F_t$ and using the tower property,
	we obtain
	$\E[S_t^{\G}\mid\mathcal F_t]
	\ge\E[X_\tau\mid\mathcal F_t]$.
	Taking the essential supremum over
	$\tau\in\cT_{t,T}^{\F}$ yields
	\eqref{eq:value}.
\end{proof}

\begin{proposition}[Equality under a common optimal stopping time]
	\label{prop:equality}
	Fix a deterministic $t\in[0,T]$.
	Suppose that some
	$\tau_t^*\in\cT_{t,T}^{\F}$
	is also optimal for the $\G$-problem,
	that is,
	$S_t^{\G}=\E[X_{\tau_t^*}\mid\mathcal G_t]$.
	Then equality holds in \eqref{eq:value}.
\end{proposition}

\begin{proof}
	By the tower property,
	$$
	\E[S_t^{\G}\mid\mathcal F_t]
	=\E[X_{\tau_t^*}\mid\mathcal F_t]
	\le S_t^{\F}.
	$$
	The reverse inequality follows from \eqref{eq:value} in 
	Theorem~\ref{thm:value}, proving equality.
\end{proof}

\begin{example}[Strict value of terminal information]
	\label{ex:strict}
	Let $X_t=B_t$ on $[0,T]$.
	Since $B$ is an $\F$-martingale, the
	optional stopping theorem gives
	$S_0^{\F}=0$.
	
	In $\G$, consider the stopping time
	$\tau=T\mathbf 1_{\{B_T>0\}}$,
	which is admissible since
	$B_T$ is $\mathcal G_0$-measurable.
	The corresponding reward is $(B_T)^+$,
	so $S_0^{\G}\ge(B_T)^+$.
	Consequently,
	\[
	\E [S_0^{\G}]
	\ge \E[(B_T)^+]
	=\sqrt{\frac{T}{2\pi}}
	>0=\E [S_0^{\F}].
	\]
	Thus the additional terminal information
	strictly increases the expected optimal
	stopping value.
\end{example}

\begin{remark}[Failure of pathwise domination]
	\label{rem:pathwise}
	The conditional comparison
	\eqref{eq:value} does not imply
	$S_t^{\G}\ge S_t^{\F}$ almost surely.
	
	To see this, let $T=1$ and $X_t=tB_t$.
	
	In $\F$, consider the stopping time
	$\tau=1/2$ on $\{B_{1/2}<0\}$ and
	$\tau=1$ otherwise. Since
	$\E[B_1\mid\mathcal F_{1/2}]=B_{1/2}$,
	the tower property gives
	\[
	\begin{aligned}
		\E[\tau B_\tau]
		=\frac12\E[B_{1/2}
		\mathbf1_{\{B_{1/2}<0\}}]
		+\E[B_1\mathbf1_{\{B_{1/2}\ge0\}}]
		&=\frac12\E[B_{1/2}
		\mathbf1_{\{B_{1/2}<0\}}]
		+\E[B_{1/2}
		\mathbf1_{\{B_{1/2}\ge0\}}]\\
		&=\frac12\E[B_{1/2}^+]
		=\frac{1}{4\sqrt{\pi}}>0.
	\end{aligned}
	\]
	Consequently,
	$S_0^{\F}\ge\E[\tau B_\tau]>0$.

Conditionally on $B_1=x$, we may write
$B_t=tx+\beta_t$, where
$\beta_t=B_t-tB_1$ is a centered Brownian
bridge independent of $B_1$.
The corresponding conditional optimal
stopping value is
$$
v(x)=
\sup_{\tau\in\cT_{0,1}^{\beta}}
\E_\beta\left[\tau^2x+\tau\beta_\tau\right],
$$
where $\cT_{0,1}^{\beta}$ denotes the
stopping times of the natural filtration
of $\beta$ taking values in $[0,1]$.
In particular, $S_0^{\G}=v(B_1)$
almost surely.

For $x<0$, we have
\[
0\le v(x)\le
\E_\beta\left[
\sup_{0\le t\le1}
(t^2x+t\beta_t)^+
\right],
\]
where $\E_\beta$ denotes expectation
with respect to the law of $\beta$.

To study the limit as $x\to-\infty$,
fix a continuous trajectory of $\beta$.
For every $\varepsilon>0$, continuity
at $0$ and $\beta_0=0$ yield
$\delta>0$ such that
$|\beta_t|<\varepsilon$ for $t\le\delta$.
Since $x<0$, we have
$(t^2x+t\beta_t)^+\le\varepsilon$
on $[0,\delta]$.
On $[\delta,1]$,
$$
t^2x+t\beta_t
\le\delta^2x+\sup_{0\le s\le1}|\beta_s|,
$$
which is negative for sufficiently
large $|x|$. Therefore,
$\sup_{0\le t\le1}(t^2x+t\beta_t)^+
\to0$ almost surely.

Moreover, for $x<0$, this supremum
is bounded by
$\sup_{0\le t\le1}|\beta_t|$,
which is integrable.
By dominated convergence,
$v(x)\to0$ as $x\to-\infty$.

	Since $S_0^{\G}=v(B_1)$ almost surely,
	there exists $M>0$ such that
	$S_0^{\G}<S_0^{\F}$ on
	$\{B_1<-M\}$, an event of positive
	probability. Thus pathwise domination
	may fail, even though the conditional
	inequality \eqref{eq:value} always holds.
\end{remark}

\begin{remark}[Value of information with restricted stopping times]
	Consider $T=1$ and the continuous,
	$\F$-adapted reward $X_t=(2t-1)B_t$.
	Suppose that stopping is restricted to
	the two dates $\{0,1\}$.
	
	In the filtration $\F$, since $X_0=0$
	and $\E[X_1]=\E[B_1]=0$, the optimal
	value at time zero is
	$$
	V_0^{\F}=\max\{0,\E[B_1]\}=0.
	$$
	In contrast, $B_1$ is already known at
	time zero in $\G$. The decision to stop
	at $1$ can therefore be made according
	to the sign of $B_1$, giving
	$$
	V_0^{\G}=\max\{0,B_1\}=(B_1)^+.
	$$
	Consequently,
	\[
	\E [V_0^{\G}]
	=\E[(B_1)^+]
	=\frac{1}{\sqrt{2\pi}}
	>0=\E[ V_0^{\F}].
	\]
	This illustrates the strict advantage
	of terminal information when stopping
	is restricted to $\{0,1\}$. It does not,
	by itself, establish strict inequality
	for the unrestricted Snell envelopes.
\end{remark}

\section{An exact conditional representation for endpoint payoffs}
\label{sec7}
The preceding sections have shown that initial
enlargement may substantially affect the
solutions of BSDEs, particularly through the
additional information available at time zero.
We now consider the special case of a zero
driver and a terminal payoff depending only
on $B_T$. In this setting, the terminal
condition is already known in $\G$, whereas
its value in $\F$ is determined by conditional
expectation. This leads to an exact projection
identity between the two solutions and, under
additional regularity, to an explicit
representation through the heat semigroup.

\begin{proposition}[Endpoint payoffs]
	\label{prop:endpoint}
	Let $h:\mathbb R\to\mathbb R$ be Borel
	measurable with $h(B_T)\in L^2$.
	For the zero-driver BSDEs, the solutions
	in $\F$ and $\G$ satisfy
	\[
	Y_t^{\F}=\E[h(B_T)\mid\mathcal F_t],
	\qquad
	Y_t^{\G}=h(B_T),\qquad Z_t^{\G}=0.
	\]
	Consequently,
	\[
	\E[Y_t^{\G}\mid\mathcal F_t]=Y_t^{\F},
	\qquad 0\le t\le T.
	\]
	If $h\in C_b^1(\mathbb R)$,
	$h\in C_b^1(\mathbb R)$\footnote{
		$C_b^1(\mathbb R)$ denotes the space of
		continuously differentiable functions
		$h:\mathbb R\to\mathbb R$ such that both
		$h$ and its derivative $h'$ are bounded.
	}
	 then
	\[
	Y_t^{\F}=P_{T-t}h(B_t),
	\qquad
	Z_t^{\F}=\partial_xP_{T-t}h(B_t),
	\]
	where $P_sh(x)=\E[h(x+\sqrt{s}N)]$
	is the heat semigroup, with
	$N\sim\mathcal N(0,1)$.
\end{proposition}

\begin{proof}
	We first consider the BSDE in the original
	filtration $\F$:
	\[
	Y_t^{\F}
	=h(B_T)-\int_t^T Z_s^{\F}\,dB_s.
	\]
	Since $h(B_T)\in L^2(\mathcal F_T)$,
	the process
	$M_t=\E[h(B_T)\mid\mathcal F_t]$
	is a square-integrable $\F$-martingale.
	By the martingale representation theorem
	for Brownian filtrations (see, e.g.,
	Karatzas and Shreve~\cite{KaratzasShreve1991}),
	there exists a unique
	$Z^{\F}\in\cH^2(\F)$ such that
	$$
	M_t=\E[h(B_T)]
	+\int_0^t Z_s^{\F}\,dB_s.
	$$
	Evaluating this identity at $T$ and
	subtracting the expression at time $t$
	gives
	$M_t=h(B_T)-\int_t^T Z_s^{\F}\,dB_s$.
	Thus $(M,Z^{\F})$ solves the zero-driver
	BSDE, and uniqueness yields
	$Y_t^{\F}=\E[h(B_T)\mid\mathcal F_t]$.
	
	We now consider the enlarged filtration
	$\G$. Since $B_T$ is $\mathcal G_0$-measurable,
	so is $h(B_T)$. Hence the constant process
	$Y_t^{\G}=h(B_T)$, together with
	$Z_t^{\G}=0$, satisfies
	$$
	Y_t^{\G}
	=h(B_T)-\int_t^T Z_s^{\G}\,dW_s.
	$$
	This solution is square-integrable.
	Its uniqueness follows from the
	well-posedness of the zero-driver BSDE
	in $\G$, since $W$ is a $\G$-Brownian
	motion with the martingale representation
	property (Theorem \ref{thm:PRP}). Consequently, the tower property
	gives
	$\E[Y_t^{\G}\mid\mathcal F_t]
	=\E[h(B_T)\mid\mathcal F_t]=Y_t^{\F}$.
	
	Finally, suppose that
	$h\in C_b^1(\mathbb R)$.
	Since $B_T=B_t+(B_T-B_t)$ and the
	increment $B_T-B_t$ is independent
	of $\mathcal F_t$ with distribution
	$\mathcal N(0,T-t)$, we obtain
	\[
	Y_t^{\F}
	=\E[h(B_t+\sqrt{T-t}\,N)\mid\mathcal F_t]
	=P_{T-t}h(B_t),
	\]
	where $N\sim\mathcal N(0,1)$ is independent
	of $\mathcal F_t$ and
	$P_sh(x)=\E[h(x+\sqrt{s}\,N)]$.

	Set $u(t,x)=P_{T-t}h(x)$.
	By the classical heat-semigroup representation
	(see Karatzas and Shreve
	\cite[Sections 4.3 \& 5.7.B]{KaratzasShreve1991}),
	$u\in C^{1,2}([0,T)\times\mathbb R)$
	and satisfies
	\[
	\partial_tu(t,x)
	+\frac12\partial_{xx}u(t,x)=0,
	\qquad t<T,
	\]
	with terminal condition $u(T,x)=h(x)$.

	Applying It\^o's formula to $u(t,B_t)$,
	and using $d\langle B\rangle_t=dt$, we obtain
	\[
	\begin{aligned}
		dY_t^{\F}
		=du(t,B_t)
		&=\left(\partial_tu(t,B_t)
		+\frac12\partial_{xx}u(t,B_t)\right)dt
		+\partial_xu(t,B_t)\,dB_t\\
		&=\partial_xu(t,B_t)\,dB_t,
	\end{aligned}
	\]
	where the last equality follows from
	the backward heat equation
	$\partial_tu+\frac12\partial_{xx}u=0$.
	
	Moreover, since $h'$ is bounded,
	differentiation under the expectation
	gives
	$\partial_xP_{T-t}h=P_{T-t}h'$.
	The resulting integrand is bounded
	and hence belongs to $\cH^2(\F)$.
	By uniqueness of the Brownian martingale
	representation,
	$$
	Z_t^{\F}
	=\partial_xP_{T-t}h(B_t)
	=P_{T-t}h'(B_t),
	$$
	for $dt\otimes d\Pp$-almost every
	$(t,\omega)$, completing the proof.
\end{proof}

\begin{remark}
	The projection identity relies on the
	zero-driver structure and the fact that
	the terminal payoff is already known
	at time zero in $\G$.
	For nonlinear generators, conditional
	projection generally does not commute
	with the generator, so the projection
	of a $\G$-BSDE need not satisfy the
	corresponding $\F$-BSDE.
\end{remark}

\begin{remark}[Projection and filtration shrinkage]
	The preceding projection identity is specific
	to zero-driver endpoint payoffs. For more
	general BSDEs, projecting a solution from
	$\G$ onto $\F$ may modify both its drift
	and martingale components.
	
	Indeed, under suitable integrability
	conditions, the optional projection of
	$A_t=\int_0^t\vartheta_s\,ds$ satisfies
	\[
	\E[A_t\mid\mathcal F_t]
	=
	\int_0^t\E[\vartheta_s\mid\mathcal F_s]\,ds
	+M_t,
	\]
	where $M$ is an $\F$-martingale.
	Thus, even the projection of a
	finite-variation process may produce
	an additional martingale term.
	
	Related calculations and applications
	to BSDEs are discussed in an unpublished
	note by Jeanblanc~\cite{JeanblancPC}.
	For nonlinear generators, conditional
	projection generally does not commute
	with the generator, so the projected
	process need not satisfy the original
	BSDE in $\F$.
\end{remark}

\section{Financial applications}
\label{sec:financial}

We illustrate the implications of initial
enlargement for financial claims through
three examples. The first concerns the value
of terminal information in an American option
exercise problem. The second distinguishes
the transfer of a reflected BSDE from
optimal stopping under enlarged information.
The third examines the martingale
representation of a European payoff.

Throughout this section, let $T>0$ and
consider a discounted risky asset satisfying
\[
dS_t=\sigma S_t\,dB_t,
\qquad S_0>0,\quad \sigma>0,
\]
where
\[
S_t=S_0\exp\left(
\sigma B_t-\frac12\sigma^2t
\right).
\]
Thus $S$ is a positive square-integrable
$\F$-martingale. We work under a fixed
probability measure $\Pp$ with zero
risk-free interest rate.

The ordinary investor observes $\F$,
whereas the informed investor observes
$\G$, where
\[
\mathcal G_t=\mathcal F_t\vee\sigma(B_T).
\]
In particular, the informed investor
knows $S_T$ at time zero.

Initial enlargement and its applications
to asymmetric financial information
have been studied in
\cite{AmendingerImkellerSchweizer1998}.
For American claims and reflected BSDEs
under asymmetric information, we refer
in particular to
\cite{EsmaeeliImkeller2018}.

For a comprehensive treatment of the
stochastic valuation of European and
American options, including martingale
pricing methods, hedging strategies,
and optimal exercise problems, we refer
to Shiryaev et al.~\cite{shiryaev1995toward}.

\subsection{Optimal exercise under
	terminal information}

Consider an American call option with
strike $K>0$ and exercise reward
\[
X_t=(S_t-K)^+,\qquad 0\leq t\leq T.
\]
For $\mathbb H\in\{\F,\G\}$, let
\[
V_t^{\mathbb H}
=
\operatorname*{ess\,sup}_{
	\tau\in\mathcal T_{t,T}^{\mathbb H}}
\E[X_\tau\mid\mathcal H_t]
\]
denote the corresponding optimal
exercise value.

\begin{proposition}
	[Strict value of terminal information]
	\label{prop:financial_call}
	
	Suppose that $0<K<S_0$. Then
	\[
	V_0^{\F}=\E[(S_T-K)^+],
	\]
	whereas
	\[
	\E[V_0^{\G}]-V_0^{\F}
	\geq
	(S_0-K)\Pp(S_T\leq K)>0.
	\]
	In particular, terminal information
	strictly increases the expected
	optimal exercise value.
\end{proposition}

\begin{proof}
	Since $S$ is an $\F$-martingale
	and $x\mapsto(x-K)^+$ is convex,
	$X$ is an $\F$-submartingale.
	The optional sampling theorem
	therefore implies
	\[
	\E[X_\tau]\leq\E[X_T],
	\qquad
	\tau\in\mathcal T_{0,T}^{\F}.
	\]
	Since $\tau=T$ is admissible,
	\[
	V_0^{\F}=\E[(S_T-K)^+].
	\]

	In the enlarged filtration $\G$, the terminal
	asset price
	\[
	S_T=S_0\exp\left(
	\sigma B_T-\frac12\sigma^2T
	\right)
	\]
	is $\mathcal G_0$-measurable, since
	$\mathcal G_0=\sigma(B_T)$ up to completion.
	Consequently, the informed investor knows
	the terminal payoff $(S_T-K)^+$ at time zero.
	
	Let
	\[
	A=\{(S_T-K)^+<S_0-K\}.
	\]
	Since $S_T$ is $\mathcal G_0$-measurable,
	we have $A\in\mathcal G_0$. Define
	\[
	\tau^*
	=0\,\mathbf 1_A+T\,\mathbf 1_{A^c}.
	\]
	This random time is a $\G$-stopping time.
	Indeed, for every $0\leq t<T$,
	\[
	\{\tau^*\leq t\}=A\in\mathcal G_0
	\subseteq\mathcal G_t,
	\]
	while $\{\tau^*\leq T\}=\Omega$.
	Thus $\tau^*\in\mathcal T_{0,T}^{\G}$.
	
	By construction, the investor exercises
	immediately on $A$, when the immediate
	payoff exceeds the terminal payoff, and
	waits until maturity on $A^c$. Therefore,
	\[
	\begin{aligned}
		X_{\tau^*}
		&=X_0\mathbf 1_A+X_T\mathbf 1_{A^c}\\
		&=(S_0-K)\mathbf 1_A
		+(S_T-K)^+\mathbf 1_{A^c}\\
		&=\max\{S_0-K,(S_T-K)^+\}.
	\end{aligned}
	\]
	Since both $S_T$ and $\tau^*$ are
	$\mathcal G_0$-measurable, the payoff
	$X_{\tau^*}$ is also
	$\mathcal G_0$-measurable. Hence
	\[
	\E[X_{\tau^*}\mid\mathcal G_0]
	=X_{\tau^*}.
	\]
	By the definition of the Snell envelope,
	\[
	\begin{aligned}
		V_0^{\G}
		&=\operatorname*{ess\,sup}_{
			\tau\in\mathcal T_{0,T}^{\G}}
		\E[X_\tau\mid\mathcal G_0]\\
		&\geq
		\E[X_{\tau^*}\mid\mathcal G_0]\\
		&=\max\{S_0-K,(S_T-K)^+\}.
	\end{aligned}
	\]
	
	Taking expectations on both sides yields
	\[
	\E[V_0^{\G}]
	\geq
	\E\left[\max\{S_0-K,(S_T-K)^+\}\right].
	\]
	Using the elementary identity
	$
	\max\{a,b\}=b+(a-b)^+,
	$
	we obtain
	\[
	\begin{aligned}
		\E[V_0^{\G}]
		&\geq \E[(S_T-K)^+]\\
		&\quad+
		\E\left[
		\big(S_0-K-(S_T-K)^+\big)^+
		\right].
	\end{aligned}
	\]
	Since
	\[
	V_0^{\F}=\E[(S_T-K)^+],
	\]
	it follows that
	\[
	\E[V_0^{\G}]-V_0^{\F}
	\geq
	\E\left[
	\big(S_0-K-(S_T-K)^+\big)^+
	\right].
	\]
	
	Now consider the event
	\[
	C=\{S_T\leq K\}.
	\]
	On $C$, we have $(S_T-K)^+=0$.
	Moreover, since $0<K<S_0$,
	\[
	\big(S_0-K-(S_T-K)^+\big)^+
	=S_0-K
	\qquad\text{on }C.
	\]
	As the integrand is nonnegative,
	\[
	\begin{aligned}
		&\E\left[
		\big(S_0-K-(S_T-K)^+\big)^+
		\right]\\
		&\quad\geq
		\E\left[
		\big(S_0-K-(S_T-K)^+\big)^+
		\mathbf 1_C
		\right]\\
		&\quad=(S_0-K)\Pp(S_T\leq K).
	\end{aligned}
	\]
	Consequently,
	\[
	\E[V_0^{\G}]-V_0^{\F}
	\geq
	(S_0-K)\Pp(S_T\leq K).
	\]
	
	Finally, since
	\[
	\log\left(\frac{S_T}{S_0}\right)
	=\sigma B_T-\frac12\sigma^2T
	\]
	and $B_T\sim\mathcal N(0,T)$,
	we obtain
	\[
	\Pp(S_T\leq K)
	=
	\Phi\left(
	\frac{\log(K/S_0)+\frac12\sigma^2T}
	{\sigma\sqrt T}
	\right)>0,
	\]
	where $\Phi$ denotes the standard normal
	distribution function.
	
	Since $S_0-K>0$, we conclude that
	$
	\E[V_0^{\G}]>V_0^{\F}.
	$
	This completes the proof.
	
\end{proof}

\begin{remark}
	The processes $V^{\F}$ and $V^{\G}$
	are the Snell envelopes of the same
	reward under different filtrations.
	They therefore correspond to the
	first components of the respective
	regular zero-driver reflected BSDEs.
	
	The preceding comparison concerns
	optimal stopping under the fixed
	probability measure $\Pp$, rather
	than arbitrage-free prices in two
	financial markets. In particular,
	advance knowledge of $S_T$ may
	change the no-arbitrage properties
	of the underlying trading model.
\end{remark}

\subsection{Early exercise and reflected
	BSDE transfer}

We next consider an exercise reward
for which the reflection mechanism
is active in the original filtration.
This example also illustrates the
difference between transferring
a reflected equation to $\G$
and solving a regular reflected
equation directly in $\G$.

Let
\[
X_t=S_t+c(T-t),\qquad c>0.
\]
The term $c(T-t)$ represents a
decreasing early-exercise premium.
This defines an American-style
exercise contract rather than a
standard vanilla call option.

\begin{proposition}
	[Active reflection and enlarged information]
	\label{prop:financial_reflection}
	
	The Snell envelope of $X$ in $\F$
	is given by
	\[
	V_t^{\F}=S_t+c(T-t).
	\]
	The corresponding zero-driver
	reflected BSDE has the unique solution
	\[
	\begin{aligned}
		Y_t^{\F}&=S_t+c(T-t),\\
		Z_t^{\F}&=\sigma S_t,\\
		K_t^{\F}&=ct.
	\end{aligned}
	\]
	
	Moreover, its transfer to $\G$
	satisfies the reflected equation
	with singular generator
	$-z\alpha_t$.
	
	By contrast, the regular
	zero-driver Snell envelope in $\G$
	satisfies
	\[
	\E[V_0^{\G}]-V_0^{\F}
	\geq
	\E[(S_T-S_0-cT)^+]>0.
	\]
	In particular,
	\[
	\Pp(V_0^{\G}>Y_0^{\F})>0.
	\]
	Thus the transferred solution
	does not coincide with the regular
	optimal stopping value in $\G$.
\end{proposition}

\begin{proof}
	Since $S$ is an $\F$-martingale,
	the process
	$X_t=S_t+c(T-t)$
	is an $\F$-supermartingale.
	Its Snell envelope therefore
	coincides with $X$.
	
	Furthermore,
	\[
	dY_t^{\F}
	=
	\sigma S_t\,dB_t-c\,dt
	=
	Z_t^{\F}\,dB_t-dK_t^{\F}.
	\]
	The terminal condition is
	$Y_T^{\F}=S_T$, and the obstacle
	condition is satisfied with equality:
	\[
	Y_t^{\F}=X_t.
	\]
	Hence
	\[
	\int_0^T(Y_t^{\F}-X_t)\,dK_t^{\F}=0.
	\]
	The uniqueness follows from the
	standard well-posedness theorem
	for reflected BSDEs in $\F$.
	
	Recall that
	\[
	dB_t=dW_t+\alpha_t\,dt,
	\qquad t<T.
	\]
	The process $Z^{\F}=\sigma S$
	has bounded continuous trajectories
	on $[0,T]$ almost surely, and
	\[
	\int_0^T|\alpha_t|\,dt<\infty
	\qquad\text{a.s.}
	\]
	Consequently,
	\[
	\int_0^T|Z_t^{\F}\alpha_t|\,dt
	<\infty\qquad\text{a.s.}
	\]
	The transferred triple therefore
	satisfies
	\[
	\begin{aligned}
		Y_t^{\F}
		={}&S_T
		-\int_t^T Z_s^{\F}\alpha_s\,ds\\
		&+K_T^{\F}-K_t^{\F}
		-\int_t^T Z_s^{\F}\,dW_s,
	\end{aligned}
	\]
	together with the same obstacle
	and Skorokhod conditions.

	On the other hand, since $S_T$ is
	$\mathcal G_0$-measurable, the informed
	investor knows the terminal payoff at
	time zero. Consider the random time
	\[
	\tau^*
	=
	\begin{cases}
		0, & S_T<S_0+cT,\\
		T, & S_T\geq S_0+cT.
	\end{cases}
	\]
	Since $\{S_T<S_0+cT\}\in\mathcal G_0$,
	it follows that $\tau^*$ is a
	$\G$-stopping time. Moreover,
	\[
	X_{\tau^*}
	=
	\max\{S_0+cT,S_T\}.
	\]
	As this payoff is $\mathcal G_0$-measurable,
	the definition of the Snell envelope yields
	\[
	\begin{aligned}
		V_0^{\G}
		&\geq
		\E[X_{\tau^*}\mid\mathcal G_0]\\
		&=\max\{S_0+cT,S_T\}.
	\end{aligned}
	\]
	Taking expectations and using
	$V_0^{\F}=S_0+cT$, we obtain
	\[
	\begin{aligned}
		\E[V_0^{\G}]-V_0^{\F}
		&\geq
		\E[\max\{S_0+cT,S_T\}]
		-(S_0+cT)\\
		&=\E[(S_T-S_0-cT)^+].
	\end{aligned}
	\]
	Finally, since $S_T$ has a lognormal
	distribution with support $(0,\infty)$,
	\[
	\Pp(S_T>S_0+cT)>0.
	\]
	Therefore,
	\[
	\E[(S_T-S_0-cT)^+]>0,
	\]
	which establishes the strict inequality.

	Finally, on the event
	$A=\{S_T>S_0+cT\}$,
	which has positive probability,
	\[
	V_0^{\G}
	\geq S_T
	>S_0+cT
	=Y_0^{\F}.
	\]
	Thus the transferred first
	component cannot coincide
	with the regular Snell envelope
	in $\G$.
\end{proof}

\begin{remark}
	This example distinguishes two
	different operations. The first
	transfers an $\F$-reflected BSDE
	to the enlarged filtration by
	introducing the singular drift
	term $-z\alpha_t$. The second
	solves a regular zero-driver
	optimal stopping problem in $\G$.
	
	Although the transferred triple
	preserves the obstacle and
	Skorokhod conditions, it need
	not represent the Snell envelope
	under the enlarged information.
	The two equations have different
	generators, and their solutions
	need not coincide.
\end{remark}

\subsection{European contingent claims
	and conditional valuation}

We finally consider a European
contingent claim, for which exercise
is restricted to maturity.
In this setting, terminal
information affects conditional
valuation and martingale
representation rather than
the choice of an exercise time.

Let
\[
\xi=(S_T-K)^+,\qquad K>0,
\]
be the payoff of a European call.

\begin{proposition}
	[Conditional valuation of a European call]
	\label{prop:financial_european}
	
	For the regular zero-driver BSDE
	in $\F$, the solution is
	\[
	Y_t^{\F}
	=
	S_t\Phi(d_1(t))-K\Phi(d_2(t)),
	\qquad t<T,
	\]
	where $\Phi$ is the standard
	normal distribution function,
	\[
	d_1(t)=
	\frac{\log(S_t/K)
		+\frac12\sigma^2(T-t)}
	{\sigma\sqrt{T-t}},
	\]
	and
	\[
	d_2(t)=d_1(t)-\sigma\sqrt{T-t}.
	\]
	The martingale integrand is
	\[
	Z_t^{\F}
	=\sigma S_t\Phi(d_1(t)).
	\]
	
	In the enlarged filtration $\G$,
	the regular zero-driver BSDE
	driven by $W$ has the solution
	\[
	Y_t^{\G}=\xi,\qquad Z_t^{\G}=0.
	\]
	Consequently,
	\[
	\E[Y_t^{\G}\mid\mathcal F_t]
	=Y_t^{\F}.
	\]
\end{proposition}

\begin{proof}
	The conditional distribution
	of $S_T$ given $\mathcal F_t$
	is lognormal. Standard Gaussian
	integration gives
	\[
	\E[(S_T-K)^+\mid\mathcal F_t]
	=
	S_t\Phi(d_1(t))-K\Phi(d_2(t)).
	\]
	Writing this expression as
	$u(t,S_t)$ and applying
	It\^o's formula yields
	\[
	dY_t^{\F}
	=
	\sigma S_t\partial_su(t,S_t)\,dB_t.
	\]
	Since
	\[
	\partial_su(t,S_t)=\Phi(d_1(t)),
	\]
	we obtain
	\[
	Z_t^{\F}
	=\sigma S_t\Phi(d_1(t)).
	\]
	
	In $\G$, the terminal payoff
	$\xi$ is $\mathcal G_0$-measurable,
	because it is a function of $B_T$.
	Therefore $(\xi,0)$ solves
	the regular zero-driver BSDE
	driven by $W$.
	
	Finally,
	\[
	\E[Y_t^{\G}\mid\mathcal F_t]
	=
	\E[\xi\mid\mathcal F_t]
	=
	Y_t^{\F}.
	\]
\end{proof}

\begin{remark}
	The preceding identity is a
	concrete financial illustration
	of the conditional projection
	result for endpoint-dependent
	payoffs.
	
	In the natural filtration,
	the martingale integrand is
	generally nonzero, whereas
	the payoff is already known
	at time zero in the enlarged
	filtration and the regular
	zero-driver representation
	has a vanishing martingale
	integrand.
	
	These conditional expectations
	should not automatically be
	interpreted as arbitrage-free
	prices or replicating values
	in an insider trading market.
\end{remark}

\section{Irregular obstacles: scope and research questions}
\label{sec8}
The preceding transfer results concern reflected
BSDEs with continuous obstacles. When the obstacle
is merely optional, the solution need not be
right-continuous, and the reflection mechanism
must account for both predictable and purely
discontinuous increasing processes. The appropriate
framework is based on the Mertens decomposition
of strong optional supermartingales; see
Grigorova et al.~\cite{GIOOQ,GIOQ}.

In a Brownian filtration, the corresponding
decomposition may be written schematically as
\[
Y_t=Y_0-\int_0^t f(s,Y_s,Z_s)\,ds
+\int_0^t Z_s\,dB_s-A_t-C_{t-},
\]
where $A$ is a predictable, right-continuous,
nondecreasing process, while $C$ is an adapted,
right-continuous, purely discontinuous
nondecreasing process. The reflection mechanism
is characterized by separate minimality
conditions governing the increments of $A$
and $C$. The complete definitions and
well-posedness theory are not reproduced here.

Under initial enlargement, the original
processes remain adapted, while the Brownian
martingale component acquires an information
drift. The pathwise reflection identities are
preserved for a transferred solution, but
the singularity of this drift raises additional
questions concerning admissibility, uniqueness,
and the characterization of the corresponding
reflected dynamics.

\begin{problem}[Transfer of irregular reflected BSDEs]
	Establish a transfer theorem for reflected BSDEs
	with optional obstacles, incorporating both
	increasing processes $A$ and $C$ in the Mertens
	decomposition. Identify integrability conditions
	ensuring that the transformed equation is
	well defined and determine appropriate
	solution classes in which uniqueness may hold
	despite the singular information drift.
\end{problem}

\begin{problem}[Projection and nonlinear Snell envelopes]
	For an optional reward process $X$ and a
	Lipschitz generator $f$, investigate conditions
	under which the $\F$-optional projection
	of the $\G$-$f$-Snell envelope retains an
	$f$-supermartingale property. More generally,
	determine whether the projected process admits
	a reflected BSDE representation with an
	explicit correction term accounting for
	the additional terminal information.
\end{problem}

\begin{problem}[Initial enlargement by a general terminal signal]
	Replace $B_T$ by a general random variable
	$L\in\mathcal F_T$. Under suitable density
	hypotheses, characterize the information
	drift and the associated innovation process.
	Determine conditions under which the
	innovation retains the predictable
	representation property, and investigate
	existence and uniqueness for the corresponding
	reflected BSDEs, including those with
	irregular obstacles.
\end{problem}

\section*{Conclusion}

The initial enlargement of the Brownian
filtration by $B_T$ reveals two distinct
features: the existence of an innovation
Brownian motion, independent of the terminal
information and possessing the predictable
representation property, and the appearance
of a singular information drift in the
semimartingale decomposition of the original
Brownian motion.

This distinction is essential for
understanding the well-posedness of BSDEs.
While regular equations driven by the
innovation process admit unique solutions
under standard assumptions, the corresponding
singular equations may exhibit nonuniqueness.
The conditional parametrization provides
a representation of regular BSDEs through
a family of Brownian equations.

For reflected BSDEs, the transfer principle
preserves the obstacle and Skorokhod
minimality conditions, although uniqueness
need not hold for the transformed generator.
From an optimal stopping perspective,
the value of terminal information is
expressed through a conditional comparison
of Snell envelopes rather than a general
pathwise domination. In the zero-driver
case, endpoint payoffs also admit an exact
conditional projection identity.

The financial applications illustrate
these results through optimal exercise
and conditional valuation under asymmetric
information. In particular, they highlight
the difference between transferring a
reflected BSDE to the enlarged filtration
and solving a regular optimal stopping
problem directly in that filtration.
The corresponding solutions need not
coincide, reflecting the different
dynamics of the two formulations.

Finally, extending the analysis to
irregular optional obstacles raises
further questions concerning Mertens
decomposition, singular-driver
integrability, and uniqueness.
The study of more general terminal
signals and nonlinear Snell envelopes
offers additional directions for
future research.

\section*{Acknowledgements}
The authors are grateful to Professor
Monique Jeanblanc for her kind communication
to Youssef Ouknine of an unpublished note
on the shrinkage of filtrations and
applications to BSDEs \cite{JeanblancPC}.

\section*{Conflict of Interest}
The authors declare that they have no conflicts of interest.

\section*{Funding}
The authors received no financial support for the research, authorship, or publication of this work.

\end{document}